\RequirePackage{ifpdf}
\ifpdf 
\documentclass[pdftex]{sigma}
\else
\documentclass{sigma}
\fi

\begin{document}

\numberwithin{equation}{section}

\allowdisplaybreaks

\renewcommand{\PaperNumber}{097}

\FirstPageHeading

\renewcommand{\thefootnote}{$\star$}

\ShortArticleName{Dif\/ferential Invariants of  Conformal and Projective Surfaces}

\ArticleName{Dif\/ferential Invariants of Conformal \\  and Projective Surfaces\footnote{This paper is a
contribution to the Proceedings of the 2007 Midwest
Geometry Conference in honor of Thomas~P.\ Branson. The full collection is available at
\href{http://www.emis.de/journals/SIGMA/MGC2007.html}{http://www.emis.de/journals/SIGMA/MGC2007.html}}}

\Author{Evelyne HUBERT~$^\dag$ and Peter J. OLVER~$^\ddag$}

\AuthorNameForHeading{E.~Hubert and P.J.~Olver}

\Address{$^\dag$~INRIA, 06902 Sophia Antipolis, France}

\EmailD{\href{mailto:Evelyne.Hubert@inria.fr}{Evelyne.Hubert@inria.fr}}
\URLaddressD{\url{http://www.inria.fr/cafe/Evelyne.Hubert}}

\Address{$^\ddag$~School of Mathematics, University of Minnesota,
Minneapolis 55455, USA}

\EmailD{\href{mailto:olver@math.umn.edu}{olver@math.umn.edu}}

\URLaddressD{\url{http://www.math.umn.edu/~olver}}

\ArticleDates{Received August 15, 2007, in f\/inal form September 24,
2007; Published online October 02, 2007}

\Abstract{We show that, for both the conformal and projective groups, all the
dif\/ferential invariants of a generic surface in three-dimensional
space can be written as combinations of the invariant derivatives of a
single dif\/ferential invariant.  The proof is based on the equivariant
method of moving frames.}

\Keywords{conformal dif\/ferential geometry;
projective dif\/ferential geometry;
dif\/ferential invariants;
moving frame;
syzygy;
dif\/ferential algebra}

\Classification{14L30; 70G65;
53A30;
53A20;
53A55;
12H05}

\renewcommand{\thefootnote}{\arabic{footnote}}
\setcounter{footnote}{0}

\section{Introduction}

According to Cartan, the local geometry of submanifolds under transformation groups, including equivalence and symmetry properties, are entirely governed by their dif\/ferential invariants.  Familiar examples are curvature and torsion of a curve in three-dimensional Euclidean space, and the Gauss and mean curvatures of a surface, \cite{Gug,E,Spivak3}.

In general, given a Lie group $G$ acting on a manifold $M$, we are interested in studying  its induced action on submanifolds $S \subset M$ of a prescribed dimension, say $p<m = \dim M$.  To this end, we prolong the group action to the submanifold jet bundles ${\rm J}^n = {\rm J}^n(M,p)$ of order $n \geq 0$,~\cite{E}.  A \textit{differential invariant} is a (perhaps locally def\/ined) real-valued function $I \colon {\rm J}^n \to {\mathbb R}$ that is invariant under the prolonged group action.    Any f\/inite-dimensional Lie group action admits an inf\/inite number of functionally independent dif\/ferential invariants of progressively higher and higher order.  Moreover, there always exist $p = \dim S$ linearly independent invariant dif\/ferential operators ${\mathcal D}_1,\dots,{\mathcal D}_p$.  For curves, the invariant dif\/ferentiation is with respect to the group-invariant arc length parameter; for Euclidean surfaces, with respect to the diagonalizing Frenet frame, \cite{Gug,KOivb,MB1,MBdicc,MBp}.
 The \textit{Fundamental Basis Theorem}, f\/irst formulated by Lie, \cite[p.~760]{LieCG},  states that all the dif\/ferential invariants can be generated from a f\/inite number of low order invariants by repeated invariant dif\/ferentiation.  A modern statement and proof of Lie's Theorem can be found, for instance, in \cite{E}.

A basic question, then, is to f\/ind a minimal set of generating dif\/ferential invariants.  For curves, where $p=1$, the answer is  known:
under mild restrictions on the group action (spe\-cif\/i\-cal\-ly transitivity and no pseudo-stabilization under prolongation), there are exactly \mbox{$m-1$} generating dif\/ferential invariants, and any other dif\/ferential invariant is a function of the generating invariants and their successive derivatives with respect to arc length  \cite{E}.
  Thus, for instance, the dif\/ferential invariants of a space curve $C \subset {\mathbb R}^3$ under the action of the Euclidean group ${\rm SE}(3)$, are generated by $m-1 = 2$ dif\/ferential invariants, namely its curvature and torsion.

In \cite{Osurf}, it was proved, surprisingly that, for generic surfaces in three-dimensional space under the action of either the Euclidean or equi-af\/f\/ine (volume-preserving af\/f\/ine) groups, a minimal system of generating dif\/ferential invariants consists of a \textit{single} dif\/ferential invariant.  In the Euclidean case, the mean curvature serves as a generator of the Euclidean dif\/ferential invariants under invariant dif\/ferentiation.  In particular, an explicit, apparently new formula expressing the Gauss curvature as a rational function of derivatives of the mean curvature with respect to the Frenet frame was found.  In the equi-af\/f\/ine case, there is a single third order dif\/ferential invariant, known as the Pick invariant, \cite{Simon,Spivak3}, which was shown to generate all the equi-af\/f\/ine dif\/ferential invariants through invariant dif\/ferentiation.

In this paper, we extend this research program to study the dif\/ferential invariants of surfaces in~${\mathbb R}^3$ under the action of the conformal and the projective groups.  Tresse classif\/ied the dif\/ferential invariants in both cases in 1894,
 \cite{Tressedi}. Subsequent developments in conformal geometry can be found in \cite{AGc,BaEaGr,FefGra, Vessiotc}, as well as the work of Tom Branson and collaborators surveyed in the papers in this special issue, while \cite{AGp,FC,MBO} present results on the projective geometry of submanifolds.

The goal of this note is to prove that, just as in the Euclidean and equi-af\/f\/ine cases, the dif\/ferential invariants of both actions are generated by a single dif\/ferential invariant though invariant dif\/ferentiation with respect to the induced Frenet frame. However, lest one be tempted to na\"ively generalize these results, \cite{Ogdi} gives examples of f\/inite-dimensional Lie groups acting on surfaces in ${\mathbb R}^3$ which require an arbitrarily large number of generating dif\/ferential invariants.
Our two main results are:

\begin{theorem}\label{conformal:th}  Every differential invariant of a
  generic surface $S \subset {\mathbb R}^3$ under the action of the conformal
  group ${\rm SO}(4,1)$ can be written in terms of a single third order
  invariant and its invariant derivatives.
\end{theorem}

\begin{theorem}\label{projective:th}   Every differential invariant of
  a generic surface $S \subset {\mathbb R}^3$ under the action of the
  projective group ${\rm PSL}(4)$ can be written in terms of a single
  fourth order invariant and its invariant derivatives.
\end{theorem}

The proofs follow the methods developed in \cite{Osurf}.
They are based on  \cite{FOmcII}, where moving frames were
 introduced as equivariant maps from the manifold to the group.
 A recent survey of the many developments and applications
this approach has entailed can be found in \cite{Ochina}.
Further extensions are in \cite{hubert07,hubert07b,hubert08a,hubert08b,Ogdi}.

A moving frame induces an invariantization process that maps dif\/ferential functions and dif\/ferential operators to dif\/ferential invariants and (non-commuting) invariant dif\/ferential operators.
Normalized  dif\/ferential invariants are
the invariantizations of the standard jet coordinates
and  are shown to generate dif\/ferential invariants at each order:
any dif\/ferential invariant can be written  as a function of
 the normalized invariants. This rewriting is actually a trivial replacement.

The key to the explicit, f\/inite description of dif\/ferential invariants
of any order lies in the \emph{recurrence formulae} that
explicitly relate the dif\/ferentiated and normalized dif\/ferential
invariants. Those formulae show that any dif\/ferential invariant
 can be written in terms of a f\/inite set of normalized dif\/ferential
invariants and their invariant derivatives.
Combined with the replacement rule, the formulae make
 the rewriting process  ef\/fective.
 Remarkably, these fundamental relations can be constructed using only the (prolonged) inf\/initesimal generators of the group action and the moving frame normalization equations.    One does \textit{not} need to know the explicit formulas for either the group action, or the moving frame, or even the dif\/ferential invariants and invariant dif\/ferential operators, in order to completely characterize
generating sets of dif\/ferential invariants and their syzygies.
Moreover 
the syzygies and recurrence relations are given by rational functions
and are thus amenable to algebraic algorithms and symbolic software
\cite{ncdiffalg,hubert05,aida, hubert07b} that we have used for this paper.

\section[Moving frames and differential invariants]{Moving frames and dif\/ferential invariants}

In this section we review the
construction of dif\/ferential invariants and invariant derivations
proposed in \cite{FOmcII}; see also \cite{hubert07b,hubert08a,Ogdi,Osurf}.
Let $G$ be an $r$-dimensional Lie group that acts (locally) on an $m$-dimensional manifold $M$.  We are interested in the
  action of $G$ on $p$-dimensional submanifolds  $N \subset M$ which, in local coordinates, we identify with the graphs of functions \mbox{$u = f(x)$}.  For each positive integer $n$, let $G^{(n)}$ denote the prolonged group action on the associated $n$-th order submanifold jet space ${\rm J}^n = {\rm J}^n(M,p)$, def\/ined as the set of equivalence classes of $p$-dimensional submanifolds of $M$ under the equivalence relation of $n$-th order contact.  Local coordinates
on ${\rm J}^n$ are denoted $z^{(n)} = (x,u^{(n)}) = (\ \dots \  x^i \ \dots \  u^\alpha _J \ \dots \ )$, with $u^\alpha _J$  representing the partial derivatives of
the dependent variables $u = (u^1,\dots,u^q)$ with respect to the independent variables $x = (x^1,\dots,x^p)$, where $p+q = m$, \cite{E}.

Assuming that the prolonged action is free\footnote{A theorem of Ovsiannikov, \cite{Ov},
slightly corrected in \cite{Osinmf}, guarantees local freeness of the prolonged action at
suf\/f\/iciently high order, provided $G$ acts locally ef\/fectively on subsets of $M$.  This is only a technical
restriction; for example, all analytic actions can be made ef\/fective by dividing by the global isotropy
subgroup.  Although all known examples of prolonged ef\/fective group actions are, in fact, free on an open subset
of a suf\/f\/iciently high order jet space, there is, frustratingly, as yet no general proof, nor known
counterexample, to this result.} on an open subset of
${\rm J}^n$, then one can construct a (locally def\/ined) \textit{moving frame}, which, according to
\cite{FOmcII}, is an equivariant map $\rho \colon V^n \to G$ def\/ined on an open subset $V^n \subset {\rm J}^n$.  Equivariance can be with respect
to either the right or left multiplication action of $G$ on itself.  All classical moving
frames, e.g., those appearing in
\cite{Cartanrm, Greenc, Griffithsmf, Gug, ivey03, Jensen}, can be regarded as left equivariant maps,
but the right equivariant versions may be easier to compute, and will be the version used here.  Of course, any right moving frame
can be converted to a left moving frame by composition with the inversion map
$g\mapsto g^{-1}$.

In practice, one constructs a moving frame by the process of normalization, relying on
the choice of a local \textit{cross-section}  $K^n \subset {\rm J}^n$ to the prolonged group orbits, meaning a submanifold of the complementary dimension that intersects each orbit transversally.  A general cross-section is prescribed implicitly by setting $r = \dim G$ dif\/ferential functions $Z = (Z_1,\dots,Z_r)$ to constants:
\begin{gather}\label{Z:eq}
Z_1(x,u^{(n)}) = c_1, \quad \dots, \quad  Z_r(x,u^{(n)}) = c_r.\end{gather}
Usually -- but not always, \cite{Mansfield,Osurf} -- the functions are selected from the jet space coordinates $x^i$, $u^\alpha _J$, resulting in a \textit{coordinate cross-section}.
The corresponding value of
the right moving frame at a jet $z^{(n)} \in {\rm J}^n$ is the unique group element $g = \rho^{(n)}(z^{(n)})\in G$ that maps it to
the cross-section:
\begin{equation}\label{mfn:eq}
\rho^{(n)}(z^{(n)})\cdot z^{(n)} = g^{(n)} \cdot z^{(n)} \in K^n.\end{equation}
  The moving frame $\rho^{(n)}$ clearly depends on the
choice of cross-section, which is usually designed so as to simplify
the required computations as much as possible.

Once the cross-section has been f\/ixed, the induced moving frame engenders an invariantization process,
 that ef\/fectively maps functions to invariants, dif\/ferential forms to invariant
dif\/ferential forms, and so on, \cite{FOmcII,Ochina}.  Geometrically, the \textit{invariantization} of any object
is def\/ined as the unique invariant object that coincides with its progenitor when restricted to the
cross-section.
In the special case of functions,
invariantization is actually entirely
def\/ined by the cross-section, and therefore
doesn't require the action to be (locally) free.
It is a projection from the ring of dif\/ferential
functions to the ring of dif\/ferential invariants,
the latter being isomorphic
to the ring of smooth functions on the cross-section  \cite{hubert07b}.

Pragmatically, the invariantization of a dif\/ferential function  is constructed by f\/irst writing out how it is transformed by the prolonged group action: $F(z^{(n)}) \mapsto F(g^{(n)} \cdot z^{(n)})$. One then
replaces  all the group parameters  by their \textit{right} moving frame formulae $g = \rho^{(n)}(z^{(n)})$, resulting in the
dif\/ferential invariant
\begin{gather}\label{iota:eq}
\iota\big[F(z^{(n)})\big] = F\big(\rho^{(n)}(z^{(n)}) \cdot z^{(n)}\big) .
\end{gather}
Dif\/ferential forms and dif\/ferential operators are handled in an analogous fashion~-- see
\cite{FOmcII,KOivb} for complete details.
Alternatively, the algebraic construction for the invariantization
of functions in \cite{hubert07b}
works with  the knowledge of the cross-section only,
i.e. without the explicit formulae for the moving frame,
and applies to  non-free actions as well.

In particular, the \textit{normalized differential invariants} induced by the moving frame
are obtained by invariantization of the basic  jet coordinates:
\begin{gather}\label{HI:eq}
H^i = \iota(x^i) ,\qquad I^\alpha_J= \iota(u^\alpha _J) ,
\end{gather}
which we collectively denote by $(H,I^{(n)}) = (\ \ldots \ H^i \ \ldots \ I^\alpha_J \ \ldots \ )$ for $\#J \leq n$.
In the case of a~coordinate cross-section,
these naturally split into two classes:  Those corresponding to the
cross-section functions $Z_\kappa$ are constant,
and known as the \textit{phantom differential
invariants}. The remainder, known as the \textit{basic differential invariants}, form a complete system of functionally independent dif\/ferential invariants.

Once the normalized dif\/ferential invariants  are known, the invariantization process \eqref{iota:eq} is implemented  by simply
replacing each jet coordinate by the corresponding normalized dif\/ferential
invariant \eqref{HI:eq}, so that
\begin{gather}\label{iotaF:eq}
\iota \big[F(x,u^{(n)})\big] = \iota\big[F(\ \dots \  x^i \ \dots \  u^\alpha _J \ \dots \ )\big]=
F(\ \dots \  H^i \ \dots \  I^\alpha_J \ \dots \ ) = F(H,I^{(n)}).\end{gather}
In particular, a dif\/ferential invariant is not af\/fected by invariantization, leading to the very useful \textit{Replacement Theorem}:
\begin{gather}\label{iotaI:eq}
J(x,u^{(n)}) = J(H,I^{(n)}) \quad \mbox{whenever $J$ is a dif\/ferential invariant.}\end{gather}
This permits one to straightforwardly rewrite any known dif\/ferential invariant
in terms the normalized invariants, and thereby establishes their completeness.

A contact-invariant coframe is obtained by taking the horizontal part
(i.e., deleting any contact forms) of the
invariantization of the basic horizontal one-forms:
\begin{gather}\label{omegai:eq}
\omega ^i \equiv  \iota(dx^i) \qquad  \mbox{modulo contact forms,}\qquad i=1,\dots,p,
\end{gather}
Invariant dif\/ferential
operators ${\mathcal D}_1,\dots,{\mathcal D}_p$  can  then be def\/ined as the associated dual
dif\/ferential operators, def\/ined so that
\[
dF  \equiv \sum_{i=1}^p ({\mathcal D}_i F) \, \omega ^i  \qquad  \mbox{modulo contact forms,}
\]
for any dif\/ferential function $F$.
Details can be found in \cite{FOmcII,KOivb}.
The invariant dif\/ferential operators  do not commute in general, but are subject to the commutation formulae
\begin{gather}\label{commutator:eq}
[{\mathcal D}_j,{\mathcal D}_k] = \sum_{i=1}^p Y^i_{jk} \,{\mathcal D}_i ,
\end{gather}
where the coef\/f\/icients $Y^i_{jk} = - Y^i_{kj}$ are certain dif\/ferential invariants
known as the \textit{commutator invariants}.

\section{Recurrence and syzygies}

In general, invariantization and dif\/ferentiation do not commute.  By a \textit{recurrence relation}, we mean an equation expressing an invariantly dif\/ferentiated invariant in terms of the basic dif\/ferential invariants.  Remarkably, the recurrence relations can be deduced knowing only the (prolonged) inf\/initesimal generators of the group action and the choice of cross-section.

Let ${\bf v}_1,\dots,{\bf v}_r$ be a basis for the inf\/initesimal generators of our transformation group. We prolong each inf\/initesimal generator to
${\rm J}^n$, resulting in the vector f\/ields
\begin{gather}\label{prv:eq}
{\bf v}^{(n)}_\kappa  =  \sum_{i=1}^p \xi ^i_\kappa (x,u) \frac{\partial}{\partial x^i}  +
\sum_{\alpha=1}^q  \sum_{j = \# J = 0}^n
\varphi ^\alpha_{J,\kappa } (x,u^{(j)}) \frac{\partial}{\partial u^\alpha _J },
\qquad \kappa=1,\dots,r,
\end{gather}
on ${\rm J}^n$. The coef\/f\/icients
$
\varphi ^\alpha _{J,\kappa } = {\bf v}^{(n)}_\kappa(u^\alpha _J)
$
are given by the prolongation formula, \cite{O,E}:
\begin{gather}\label{phiaJ:eq}
\varphi ^\alpha_{J,\kappa}
= D_J \left(\varphi ^\alpha_\kappa - \sum_{i = 1}^p \xi ^i_\kappa \,u^\alpha_i\right)
 + \sum_{i=1}^p \xi^i_\kappa u^\alpha_{J,i}
,
\end{gather}
where $D_1,\dots,D_p$ are the usual (commuting) total derivative operators, and $D_J = D_{j_1}  \cdots   D_{j_k}$ the corresponding iterated total derivative.

Given  a collection $F = (F_1,\dots,F_k)$ of dif\/ferential functions, let
\begin{gather}\label{vF:eq}
{\bf v}(F) = \big({\bf v}^{(n)}_\kappa(F_j)\big)\end{gather}
denote the $r \times k$ \textit{generalized Lie matrix} obtained by applying the prolonged inf\/initesimal generators to the dif\/ferential functions.  In particular, $L^{(n)}(x,u^{(n)}) = {\bf v}(x,u^{(n)})$ is the classical Lie matrix of order $n$ whose entries are the inf\/initesimal generator coef\/f\/icients $\xi ^i_\kappa$, $\varphi ^\alpha_{J,\kappa }$,~\cite{E,Ogdi}. The rank of the classical Lie matrix $L^{(n)}(x,u^{(n)})$ equals the dimension of the prolonged group orbit passing through the point $(x,u^{(n)}) \in {\rm J}^n$.  We set
\begin{gather}\label{rn:eq}
r_n = \max \big\{{\rm rank}\, L^{(n)}(x,u^{(n)})\,|\,(x,u^{(n)}) \in {\rm J}^n\big\}
\end{gather}
to be the maximal prolonged orbit dimension. Clearly, $r_0 \leq r_1 \leq r_2 \leq \cdots   \leq r = \dim G$, and $r_n = r$ if and only if the action is locally free on an open subset of ${\rm J}^n$.   Assuming $G$ acts locally ef\/fectively on subsets, \cite{Osinmf}, this holds for $n$ suf\/f\/iciently large. We def\/ine the \textit{stabilization order} $s$ to be the minimal $n$ such that $r_n = r$.  Locally, the number of functionally independent dif\/ferential invariants of order $\leq n$ equals $\dim {\rm J}^n - r_n$.

The fundamental moving frame recurrence formulae were f\/irst established in \cite{FOmcII} and written as follows; see also \cite{Ogdi} for additional details.

\begin{theorem}\label{recurrence:th} The \textit{recurrence formulae} for the normalized differential invariants have the form
\begin{gather}\label{DHI:eq}
{\mathcal D}_iH^j = \delta ^j_i + \sum_{\kappa=1}^r R_i^\kappa  \, \iota(\xi ^j_\kappa ) ,\qquad
{\mathcal D}_iI^\alpha _J = I^\alpha _{Ji} + \sum_{\kappa=1}^r R_i^\kappa \, \iota(\varphi ^\alpha _{J,\kappa }),
\end{gather}
where $\delta ^j_i$ is the usual Kronecker delta, and $R_i^\kappa$ are certain differential invariants.
\end{theorem}

The recurrence formulae \eqref{DHI:eq} imply the following commutator syzygies among the normalized dif\/ferential invariants:
\begin{gather}\label{comsyz:eq}
{\mathcal D}_iI^\alpha _{Jj} - {\mathcal D}_jI^\alpha _{Ji} = \sum_{\kappa=1}^r \big[R_i^\kappa \, \iota(\varphi ^\alpha _{Jj,\kappa }) - R_j^\kappa \, \iota(\varphi ^\alpha _{Ji,\kappa })\big],\end{gather}
for all $1 \leq i,j \leq p$ and all multi-indices $J$.
We can show that a subset of these relationships~\eqref{DHI:eq},~\eqref{comsyz:eq} form a
complete set of syzygies, \cite{hubert08a}. By formally manipulating those syzygies, performing dif\/ferential
elimination \cite{diffalg,hubert03d,ncdiffalg,hubert05},
we are able to obtain expressions of
some of the dif\/ferential invariants in terms of the invariant derivatives
of others. This is the strategy for the main results of this paper.

In the case of coordinate cross-section, if we single out the recurrence formulae for the constant \textit{phantom differential invariants} prescribed by the cross-section, the left hand sides are all zero, and hence we obtain a linear algebraic system that can be uniquely solved for the invariants~$R_i^\kappa$.
Substituting the resulting formulae back into the recurrence formulae for the remaining, non-constant basic dif\/ferential invariants leads to a complete system of relations among the normalized dif\/ferential invariants \cite{FOmcII,Ogdi}.

More generally, if we think of the $R_i^\kappa$ as the entries of a $p\times r$ matrix
\begin{gather}\label{mcm:eq}
R = (R_i^\kappa ),
\end{gather}
then they are given explicitly by
\begin{gather}\label{MCm:eq}
R = -\iota\big[D(Z)\,{\bf v}(Z)^{-1}\big],\end{gather}
where $Z = (Z_1,\dots,Z_r)$ are the cross-section functions \eqref{Z:eq}, while
\begin{gather}\label{DZ:eq}
D(Z) = (D_i Z_j)\end{gather}
is the $p \times r$ matrix of their total derivatives.
The recurrence formulae are then covered by the matricial equation \cite{hubert08a}
\begin{gather}\label{urf:eq}
{\mathcal D}(\iota(F)) = \iota(D(F)) + R\>\iota({\bf v}(F)),\end{gather}
for any set of dif\/ferential functions $F = (F_1,\dots,F_k)$.
The left hand side denotes the $p \times k$ matrix
\begin{gather}\label{CDF:eq}
{\mathcal D}(\iota(F)) = ({\mathcal D}_i(\iota(F_j)))
\end{gather}
obtained by invariant dif\/ferentiation.

The invariants  $R_i^\kappa$  actually arise  in the proof of \eqref{DHI:eq}
 as the coef\/f\/icients of the horizontal parts of the pull-back of the  Maurer--Cartan forms via the moving frame, \cite{FOmcII}.  Explicitly, if $\mu^1,\dots,\mu^r$ are a basis for the Maurer--Cartan forms on $G$ dual to the Lie algebra basis ${\bf v}_1,\dots,{\bf v}_r$, then the horizontal part of their pull-back by the moving frame can be expressed in terms of the contact-invariant coframe \eqref{omegai:eq}:
\begin{gather}\label{mfu:eq}
  \gamma^\kappa= \rho^* \mu^\kappa  \equiv \sum_{i=1}^p R_i^\kappa \,\omega ^i  \qquad \mbox{modulo contact forms.}
  \end{gather}
We shall therefore refer to $R_i^\kappa$ as the \textit{Maurer--Cartan invariants}, while $R$ in \eqref{mcm:eq} will be called the \textit{Maurer--Cartan matrix}.
  In the case of curves,  when $G \subset {\rm GL}(N)$ is a matrix Lie group,
the Maurer--Cartan matrix $R =  {\mathcal D} \rho^{(n)}(x,u^{(n)}) \cdot \rho^{(n)}(x,u^{(n)})^{-1}$ can be identif\/ied with the
 Frenet--Serret matrix, \cite{Gug, MBp}, with ${\mathcal D}$ the invariant arc-length derivative.

The identif\/ication \eqref{mfu:eq} of the Maurer--Cartan invariants as
the coef\/f\/icients of the (horizontal parts of) the pulled-back
Maurer--Cartan forms can be used to deduce their syzygies,
\cite{hubert08b}.  The Maurer--Cartan forms on $G$ satisfy the usual Lie
group structure equations
\begin{gather}\label{mcseq:eq}
d \mu^c = - \sum_{a<b} C_{ab}^c \,\mu^a \wedge \mu^b,\qquad c=1,\dots,r,\end{gather}
where $C_{ab}^c$ are the structure constants of the Lie algebra relative to the basis ${\bf v}_1,\dots,{\bf v}_r$.
It follows that their pull-backs \eqref{mfu:eq} satisfy the same equations:
\begin{gather}\label{hmcseq:eq}
d \gamma^c = - \sum_{a<b} C_{ab}^c\, \gamma^a \wedge \gamma^b,\qquad c=1,\dots,r.
\end{gather}
The purely horizontal components of these identities provide the following syzygies among the Maurer--Cartan invariants, \cite{hubert08b}:

\begin{theorem}\label{mcsyz:th}
 The Maurer--Cartan invariants
  satisfy the following identities:
\begin{gather}\label{MCsyz:eq}
 {\mathcal D}_j(R^i_c) - {\mathcal D}_i(R^j_c)
  + \sum_{1\leq a<b\leq r}
      C_{ab}^c\, (R^i_a R^j_b- R^j_aR^i_b)
    + \sum_{k=1}^p Y^i_{jk}  \, R^k_c  = 0, \end{gather}
 for
$1\leq c\leq r$, $1\leq i<j\leq p$,  and where
$Y^i_{jk} $ are the commutator invariants \eqref{commutator:eq}.
\end{theorem}

Finally, we note the recurrence formulas for the invariant dif\/ferential forms established in \cite{FOmcII} produce the explicit formulas for the commutator invariants:
\begin{gather}\label{Y:eq}
Y^i_{jk} = \sum_{\kappa=1}^r \sum_{j=1}^p
R_j^\kappa\,\iota (D_j\xi ^i_\kappa) - R^\kappa_k\,\iota (D_k\xi ^i_\kappa).
\end{gather}

\section[Generating differential invariants]{Generating dif\/ferential invariants}

A set of dif\/ferential invariants ${\mathfrak I} = \{I_1,\dots,I_k\}$ is called \textit{generating} if, locally, every dif\/ferential invariant can be expressed as a function of them and their iterated invariant derivatives ${\mathcal D}_J I_\nu$.  A key issue is to f\/ind a minimal set of generating invariants, which (except for curves) must be done on a case by case basis.  Before investigating the minimality question in the conformal and projective examples, let us state general results characterizing (usually non-minimal) generating systems.  These results are all consequences of the recurrence formulae, \eqref{DHI:eq} or \eqref{urf:eq}, that furthermore make the rewriting  constructive.

 Let
\begin{gather}\label{ninv:eq}
{\mathfrak J}^n =  \{ H^1, \ldots, H^p\}
  \>\cup\> \{ I^\alpha_J\, |\, \alpha =1,\dots,q, \#J \leq n\} \end{gather}
  denote the complete set of normalized dif\/ferential invariants of order $\leq n$.  In particular, assuming we choose a cross-section that projects to a cross-section on $M$ (e.g., a minimal order cross-section) then ${\mathfrak J}^0 = \{H^1, \ldots, H^p,I^1,\ldots, I^q\}$ are the ordinary invariants for the action on~$M$.  In particular, if, as in the examples treated here, the action is transitive on $M$, the normalized order~$0$ invariants are all constant, and hence are superf\/luous for the following generating systems.

\begin{theorem}\label{cngen:th}
If the moving frame has order $n$, then
the set of normalized differential invariants
${\mathfrak J}^{n+1}$ of order $n+1$ or less forms a
generating set.
\end{theorem}

For cross-section of \textit{minimal order} there is an additional
important  set of invariants that is generating. This was proved for coordinate cross-sections in
\cite{Ogdi} and then generalized
in \cite{hubert08a}.  For each $k \geq 0$, let $r_k$ denote the maximal orbit dimension of the action of $G^{(k)}$ on ${\rm J}^k$.

\begin{theorem}\label{minorder:th}
Let $Z= (Z_1,\dots,Z_r)$ define
a \textit{minimal order cross-section} in the sense that for each $k = 0, 1, \ldots, s$, where $s$ is the stabilization order,
$Z_k=(Z_1, \ldots, Z_{r_k})$ defines a cross-section
for the action of $G^{(k)}$ on ${\rm J}^k$.
 Then ${\mathfrak J}^0 \> \cup \>\mathfrak{Z}$, where
\begin{gather}\label{iZ:eq}
\mathfrak{Z}=
\{\iota (D_i(Z_j)) \,|\, 1\leq i\leq p, \; 1\leq j\leq r \},\end{gather}
form a generating set of differential invariants.
\end{theorem}

Another interesting consequence of Theorem \ref{recurrence:th} observed
in \cite{hubert08b} is
that the Maurer--Cartan invariants
\begin{gather}\label{mcinv:eq}
{\mathfrak R}= \{ R^i_a \,|\, 1\leq i\leq p, \; 1\leq a\leq r\}\end{gather}
also form a generating set when the action is
transitive on $M$. More precisely:

\begin{theorem}\label{mcgen:th} The differential invariants ${\mathfrak J}^0\> \cup \>{\mathfrak R}$ form a generating set.
\end{theorem}

In \cite{Osurf}, the following device for generating the commutator invariants was introduced, and then applied to the dif\/ferential invariants of Euclidean and equi-af\/f\/ine surfaces.  We will employ the same trick here.

\begin{theorem}\label{ci:th} Let $I = (I_1,\dots,I_p)$ be a set of differential invariants such that ${\mathcal D}(I)$, cf.~\eqref{CDF:eq}, forms a nonsingular $p\times p$ matrix of differentiated invariants.  Then one can express the commutator invariants as rational functions of the invariant derivatives, of order $\leq 2$, of $I_1,\dots,I_p$.
\end{theorem}

\begin{proof}
In view of \eqref{commutator:eq}, we have
\begin{gather}\label{DijIk:eq}
{\mathcal D}_i {\mathcal D}_j I_l - {\mathcal D}_j {\mathcal D}_i I_l = \sum_{k=1}^p Y^i_{jk} \,{\mathcal D}_k I_l.\end{gather}
We regard \eqref{DijIk:eq} as a system of $p$ linear equations for the
commutator invariants $Y^i_{j1}, \ldots, Y^i_{jp}$.  Our assumption
implies that the coef\/f\/icient matrix is  nonsingular.  Solving the linear
system by, say, Cramer's rule, produces the formulae for the $Y^i_{jk}$.
\end{proof}

In particular, if $I$ is any single dif\/ferential invariant with suf\/f\/iciently many nontrivial invariant derivatives,
the dif\/ferential invariants in the proposition can be taken as invariant derivatives of~$I$.
Typically we choose $I$ of order at least $n$, the order of the moving frame, and $p-1$ of its f\/irst order invariant derivatives.
If $I$ is a basic invariant, nonsingularity of the matrix of dif\/ferentiated invariants is then a consequence of the recurrence formulae.
As a result, one is, in fact, able to generate all of the commutator invariants as combinations of derivatives of a \textit{single differential invariant\/}!

\section[Differential invariants of surfaces]{Dif\/ferential invariants of surfaces}

Let us specialize the preceding general constructions to the case of two-dimensional surfaces in three-dimensional space.  Let $G$ be a $r$-dimensional Lie group acting transitively and ef\/fectively on $M = {\mathbb R}^3$.
Let ${\rm J}^n = {\rm J}^n({\mathbb R}^3,2)$ denote the $n$-th order surface jet bundle, with the usual induced coordinates $z^{(n)} = (x,y,u,u_x,u_y,u_{xx}, \dots , u_{jk}, \dots )$ for $j+k \leq n$.

Let $n \geq s$, the stabilization order of $G$.  Given a cross-section $K^n \subset {\rm J}^n$, let $\rho\: \colon V^n \to G$ be the induced right moving frame def\/ined on a suitable open subset $V^n \subset {\rm J}^n$ containing $K^n $.
Invariantization of the basic jet coordinates results in the \textit{normalized differential invariants}
\begin{gather}\label{ndi:eq}
H_1 = \iota(x),\qquad H_2 = \iota(y),\qquad I_{jk} = \iota(u_{jk}),\qquad  j, k \geq 0.
\end{gather}
In view of our transitivity assumption, we will only consider cross-sections that normalize the order $0$ variables, $x = y = u = 0$, and so the order $0$ normalized invariants are trivial: $H_1 = H_2 = I_{00} = 0$.
We use
\begin{gather}\label{In:eq}
I^{(n)} = (0,I_{10},I_{01},I_{20},I_{11}, \dots  ,I_{0n}) = \iota(u^{(n)})\end{gather}
to denote all the normalized dif\/ferential invariants, both phantom and basic, of order $\leq n$ obtained by invariantizing the dependent variable $u$ and its derivatives.

In addition, the two invariant dif\/ferential operators
 ${\mathcal D}_1$, ${\mathcal D}_2$, are obtained
as the total derivations  dual
to  the contact-invariant coframe determined by the moving frame:
Specializing the general moving frame recurrence formulae in Theorem~\ref{recurrence:th}, we have:

\begin{theorem}\label{rf:th} The \textit{recurrence formulae} for the differentiated invariants are
\begin{gather}\label{rf:eq}
{\mathcal D}_1 I_{jk} = I_{j+1,k} + \sum_{\kappa=1}^r \varphi _\kappa ^{jk}(0,0,I^{(j+k)}) R^\kappa_1,\nonumber\\
{\mathcal D}_2 I_{jk} = I_{j,k+1} + \sum_{\kappa=1}^r \varphi _\kappa ^{jk}(0,0,I^{(j+k)}) R^\kappa_2 ,
\qquad j+k \geq 1,
\end{gather}
where $R^\kappa_i$ are the Maurer--Cartan invariants, which multiply the invariantizations of the coefficients of the prolonged infinitesimal generator
\begin{gather}\label{v:eq}
{\bf v}_{\kappa} = \xi_\kappa (x,y,u) \frac{\partial}{\partial x} + \eta_\kappa(x,y,u) \frac{\partial}{\partial y}  + \sum_{0 \leq j+k \leq n} \varphi _\kappa^{jk}(x,y,u^{(j+k)})\frac{\partial}{\partial u_{jk}},\end{gather}
which are given explicitly by the usual prolongation formula \eqref{phiaJ:eq}:
\begin{gather}\label{vnc:eq}
\varphi_\kappa ^{jk} = D_x^jD_y^k (\varphi _\kappa- \xi_\kappa\,  u_x - \eta _\kappa\, u_y) + \xi _\kappa\, u_{j+1,k} + \eta _\kappa\, u_{k,j+1}.\end{gather}
\end{theorem}

\section{Surfaces in conformal geometry}

In this section, we focus our attention on the standard action of the conformal group ${\rm SO}(4,1)$
on surfaces in ${\mathbb R}^3$, \cite{AGc}. Note that $\dim \hbox{SO}(4,1) = 10$.
A basis for its inf\/initesimal generators is
\begin{gather*}
 \frac{\partial}{\partial x}, \qquad \frac{\partial}{\partial y},\qquad \frac{\partial}{\partial u}, \qquad
 x\frac{\partial}{\partial y}-y\frac{\partial}{\partial x}, \qquad x\frac{\partial}{\partial u}-u\frac{\partial}{\partial x},\qquad y\frac{\partial}{\partial u}-u\frac{\partial}{\partial y}, \\ x\frac{\partial}{\partial x}+y\frac{\partial}{\partial y}+ u\frac{\partial}{\partial u},\qquad
(x^2-y^2-u^2) \frac{\partial}{\partial x}+2xy\frac{\partial}{\partial y}+2x u \frac{\partial}{\partial u},\\
   2 xy\frac{\partial}{\partial x}+ (y^2-x^2-u^2) \frac{\partial}{\partial y}  +2yu \frac{\partial}{\partial u},\qquad
   2xu\frac{\partial}{\partial u}+2yu \frac{\partial}{\partial y}+ (u^2-x^2-y^2) \frac{\partial}{\partial u}.
\end{gather*}

The maximal prolonged orbit dimensions \eqref{rn:eq} are $r_0=3$, $r_1=5$, $r_2=8$ and $r_3=10$.  The stabilization order is thus $s=3$.
The action is transitive on an open subset of ${\rm J}^2$ and there
are two independent dif\/ferential invariants of order $3$.
Thus, by~Theorem~\ref{cngen:th}, the dif\/ferential invariants of
order $3$ and $4$ form a generating set.
In this section we shall show that, under a certain non-degeneracy condition,
 all the dif\/ferential invariants can be written in terms of the
 derivatives of a single third order dif\/ferential invariant.

The argument goes in two steps. We f\/irst show that all the
dif\/ferential invariants of fourth order can be written in terms of
the two third order dif\/ferential invariants and their monotone derivatives, i.e.~those
obtained by applying the operators ${\mathcal D}_1^i {\mathcal D}_2^j$. Then, the commutator trick of Theorem~\ref{ci:th} allows us to reduce to a single generator.

We give two computational proofs of the f\/irst step. First
using the properties of normalized invariants,
Theorems~\ref{recurrence:th}~and~\ref{minorder:th},
and a cross-section that corresponds to a hyperbolic quadratic form,
second by using the properties of the Maurer--Cartan invariants,
Theorems~\ref{mcgen:th} and~\ref{mcsyz:th}, along with a cross-section that corresponds to a
degenerate quadratic form.
We have used the symbolic computation
software \textsc{aida} \cite{aida} to compute the
Maurer--Cartan matrix, the commutation rules and the
syzygies, and the software \emph{diffalg} \cite{diffalg,ncdiffalg}
 to operate the dif\/ferential elimination.

\subsection{Hyperbolic cross-section}

The cross-section implicitly used in \cite{Tressedi} is:
\begin{gather}\label{hypcs:eq}
x=y=u=u_x=u_y=u_{xx}=u_{yy}=u_{xxy}=u_{xyy}=0,\qquad u_{xy} = 1.
\end{gather}

Thus, there are two basic third order dif\/ferential invariants:
\[
I_{30} = \iota(u_{xxx}),\qquad I_{03} = \iota(u_{yyy}),
\]
and $5$ of order $4$, given by invariantization of the fourth order jet coordinates: $I_{jk} = \iota(u_{jk})$, $j+k=4$.  Since \eqref{hypcs:eq} def\/ines a minimal order cross-section, Theorem~\ref{minorder:th} implies that
$\{I_{30}, I_{03},I_{31},I_{22},I_{13}\}$ is a generating set of dif\/ferential invariants.

To prove Theorem~\ref{conformal:th}, we f\/irst show that $I_{31}$, $I_{13}$ and $I_{22}$ can be written in terms of
$\{I_{30},I_{03}\}$ and their monotone derivatives.
Using formula~\eqref{MCm:eq}, the Maurer--Cartan matrix is found to have the form
\[
R =
- \begin{pmatrix}
 1&0&0&\phi&0&1&0&\kappa&\sigma&\phi
\\
0&1&0&\psi&1&0&0&\sigma&\tau&-\psi\end{pmatrix}
\]
where
\begin{gather*}
 \phi=-\tfrac{1}{4}\,I_{{30}},\qquad
\psi=\tfrac{1}{4}\,I_{{03}},\qquad \tau=1-\tfrac{1}{2}\,I_{{13}}-\tfrac{1}{8}\,{I_{{03}}}^{2},\\
\sigma=\tfrac{1}{8}\,I_{{30}}I_{{03}}-\tfrac{1}{2}\,I_{{22}},\qquad
\kappa=1-\tfrac{1}{2}\,I_{{31}}-\tfrac{1}{8}\,{I_{{30}}}^{2}.
\end{gather*}
The f\/irst two are, in fact, the commutator invariants since, by \eqref{Y:eq}, the invariant derivations~${\mathcal D}_1$ and ${\mathcal D}_2$ satisfy the
commutation rule:
\begin{gather}\label{comm2:eq}
   [{\mathcal D}_2,{\mathcal D}_1] = \phi\, {\mathcal D}_1 + \psi\,{\mathcal D}_2 .
\end{gather}
Implementing \eqref{DHI:eq},~\eqref{comsyz:eq},
we deduce the following relationships among
$\{I_{30}, I_{03}, I_{40}, I_{31},I_{22}$, $I_{13},I_{04}\}$:
\begin{alignat*}{3}
& {E_{301}}: \quad && {\mathcal D}_1(I_{30})-3\,{I_{22}}
+\tfrac{3}{4}\,{I_{30}}\,{I_{03}}-{I_{40}},&
\\
& {E_{302}}: && {\mathcal D}_2(I_{30})-3\,{I_{13}}
-\tfrac{3}{4}\,{{I_{03}}}^{2}+6-{I_{31}},&
\\
& {E_{031}}: && {\mathcal D}_1(I_{03})-3\,{I_{31}}
-\tfrac{3}{4}\,{{I_{30}}}^{2}+6-{I_{13}},&
\\
&{E_{032}}: && {\mathcal D}_2(I_{03})-3\,{I_{22}}
+\tfrac{3}{4}\,{I_{30}}\,{I_{03}}-{I_{04}},&
\\
& {S_{14}}: && {\mathcal D}_2(I_{13})-{\mathcal D}_1(I_{04})
+\tfrac{3}{4}\,{I_{03}}\,{I_{22}}
-\tfrac{1}{4}\,{I_{03}}\,{I_{04}}+{I_{30}}\,{I_{13}},&
\\
&{S_{23}}: && {\mathcal D}_2(I_{22})-{\mathcal D}_1(I_{13})
-\tfrac{3}{2}\,{I_{03}}\,({I_{31}}+{I_{13}})
-\tfrac{1}{4}\,{I_{30}}\,({I_{22}}+{I_{04}})
-\tfrac{1}{4}\,{I_{03}}\,({{I_{30}}}^{2}+{{I_{03}}}^{2}-20),&
\\
&{S_{32}}: && {\mathcal D}_2(I_{31})-{\mathcal D}_1(I_{22})
+\tfrac{1}{4}\,{I_{03}}\,( {I_{40}}+{I_{22}})
+\tfrac{3}{2}\,{I_{30}}\,({I_{13}}+{I_{31}})
+\tfrac{1}{4}\,{I_{30}}\,({{I_{03}}}^{2}
         +{{I_{30}}}^{2}-20),&
\\
&{S_{41}}: &&{\mathcal D}_2(I_{40})-{\mathcal D}_1(I_{31})
-{I_{03}}\,{I_{31}}
-\tfrac{3}{4}\,{I_{30}}\,{I_{22}}
+\tfrac{1}{4}\,{I_{30}}\,{I_{40}}.
&
\end{alignat*}
Taking the combination
$E_{302}-3\,E_{031}$
and
$E_{031}-3\,E_{302}$
we obtain:
\begin{gather*}
I_{31}
= \tfrac{3}{2}-\tfrac{1}{8}\,{\mathcal D}_2(I_{30})
+\tfrac{3}{8}\,{\mathcal D}_1(I_{03})
+{\tfrac{3}{32}}\,{(I_{03})}^{2}
-{\tfrac{9}{32}}\,{(I_{30})}^{2}
,
\\
I_{13}
=
\tfrac{3}{2}-\tfrac{1}{8}\,{\mathcal D}_1(I_{03})
+\tfrac{3}{8}\,{\mathcal D}_2(I_{30})
-{\tfrac{9}{32}}\,{(I_{03})}^{2}
+{\tfrac{3}{32}}\,{(I_{30})}^{2}
.
\end{gather*}
Taking the combination
\begin{gather*}
128\,{\mathcal D}_2(S_{32})-48\,{\mathcal D}_1(S_{41})-16\,{\mathcal D}_1(S_{23})
-36\,I_{03}S_{41}-12\,I_{03}S_{23}+108\,I_{30}S_{32} +4\,I_{30}S_{14}
\\
\qquad{}
-48\,{\mathcal D}_1{\mathcal D}_2(E_{301})-16\,{\mathcal D}_2^{2}(E_{302}) +48\,{\mathcal D}_2^{2}(E_{031})+16\,{\mathcal D}_1^{2}(E_{031})
\\
\qquad{}+36\,I_{03}{\mathcal D}_1(E_{031})+88\,I_{30}{\mathcal D}_2(E_{031})
-12\,I_{30}{\mathcal D}_1(E_{301})-4\,I_{03}{\mathcal D}_2(E_{301})
\\
\qquad{}
+36\,I_{30}{\mathcal D}_2(E_{302})+\left( 18\,{I_{03}}^{2}
       +40\,{I_{30}}^{2}+48\,{\mathcal D}_2(I_{30})+24\,{\mathcal D}_1(I_{03})
     \right)    E_{031}
\\
\qquad{}
+ \left( 18\,I_{30}I_{03}-12\,{\mathcal D}_1(I_{30})+32\,{\mathcal D}_2(I_{03}) \right) E_{301}
\\
\qquad{}
+ \left( 42\,{I_{30}}^{2}+48\,{\mathcal D}_2(I_{30}) \right) E_{302}
+ \left( 2\,I_{30}I_{03}+4\,{\mathcal D}_1(I_{30}) \right)E_{032}
\end{gather*}
leads to:
\[
I_{22} = 1-\frac{A_{22}}{64\,B_{22}},
\]
where
\begin{gather*}
A_{22} =
64\,{\mathcal D}_2^{3}(I_{30})
-\,48\,{\mathcal D}_1^{2}{\mathcal D}_2(I_{30})-48\,{\mathcal D}_1{\mathcal D}_2^{2}(I_{03})
-64\,{\mathcal D}_1^{3}(I_{03}) \\
\phantom{A_{22} =}{} + \left( 36\,{\mathcal D}_1^{2}(I_{03})
            +48\,{\mathcal D}_2^{2}(I_{03})
            -52\,{\mathcal D}_1{\mathcal D}_2(I_{30})
        \right) I_{03}
\\
\phantom{A_{22} =}{}- \left(
36\,{\mathcal D}_2^{2}(I_{30})+24\,{\mathcal D}_1^{2}(I_{30})
-28\,{\mathcal D}_1{\mathcal D}_2(I_{03}) \right) I_{30}
\\
\phantom{A_{22} =}{} +
36\,{{\mathcal D}_2(I_{03})}^{2}
-24\,{{\mathcal D}_1(I_{30})}^{2}+24\,{{\mathcal D}_1(I_{03})}^{2}-24\,{{\mathcal D}_2(I_{30})}^{2}
-12\,{\mathcal D}_2(I_{30}){\mathcal D}_1(I_{03})
\\
\phantom{A_{22} =}{} +
\left(30\,{\mathcal D}_1(I_{03}) -8\,{\mathcal D}_2(I_{30}) \right) {I_{03}}^{2}
+ \left(52\,{\mathcal D}_2(I_{03}) -42\,{\mathcal D}_1(I_{30}) \right)
 I_{30}{ I_{03}}
\\
\phantom{A_{22} =}{} - \left(30\,{\mathcal D}_2(I_{30})+2\,{\mathcal D}_1(I_{03}) \right)
{{ I_{30}}}^{2}
+ 3\,{I_{03}}^{4}-3\,I_{30}^{4}+3\,{I_{03}}^{2}-3\,I_{30}^{2},
\end{gather*}
and
\[
B_{22} = {{\mathcal D}_1(I_{30})-{\mathcal D}_2(I_{03})}.
\]

We conclude that the two third order invariants  $I_{30}$ and $I_{03}$ form a generating system.  Moreover, since the generating invariants are, up to constant multiple, commutator invariants, we can use the commutator trick of Theorem~\ref{ci:th} to generate them both from any single dif\/ferential invariant.  Indeed, when  ${\mathcal D}_2\phi\neq 0$
the commutation rule \eqref{comm2:eq} implies
that
\begin{gather}\label{psi:eq}
\psi = \frac{{\mathcal D}_2{\mathcal D}_1\phi
              -{\mathcal D}_1{\mathcal D}_2\phi -\phi{\mathcal D}_1\phi}{{\mathcal D}_2\phi}.
\end{gather}
Similarly, when ${\mathcal D}_1\psi\neq 0$ we have
\begin{gather}\label{phi:eq}
\phi = \frac{{\mathcal D}_2{\mathcal D}_1\psi
              -{\mathcal D}_1{\mathcal D}_2\psi-\psi{\mathcal D}_2\psi}{{\mathcal D}_1\psi}.
\end{gather}
Therefore, under the assumption that
\begin{gather}\label{p1p20:eq}
({\mathcal D}_1\psi)^2+({\mathcal D}_2\phi)^2\neq 0,
\end{gather}
a single dif\/ferential invariant, of order~3,
generates all the dif\/ferential invariants for surfaces in conformal geometry.

\subsection{Degenerate cross-section}

In our second approach, we choose the ``degenerate'' cross-section
\begin{gather}\label{degcs:eq}
x=y=u=u_x=u_y=u_{xx}=u_{xy} = u_{yy}=u_{xxy}=u_{xyy}=0.
\end{gather}
Implementing \eqref{MCm:eq}, the new Maurer--Cartan matrix is:
\[
R =
-\begin{pmatrix}
 1&0&0&0&1&0&-\psi&\sigma&\kappa&0
\\ 0&1&0&0&0&0&\phi&\tau&-\sigma&-\tfrac 12\,\phi\end{pmatrix},
\]
where
\[ \phi=I_{{03}},\qquad \psi=I_{{30}},\qquad\tau=\tfrac{1}{2}\,I_{{13}},
 \qquad \kappa=-\tfrac{1}{2}\,I_{{31}},\qquad\sigma=\tfrac{1}{2}\,I_{{22}}.
\]
 Again, $\phi$, $\psi$ are the commutator invariants since $
   [{\mathcal D}_2,{\mathcal D}_1] = \phi\, {\mathcal D}_1 + \psi\,{\mathcal D}_2 $.
Theorem~\ref{mcgen:th} tells us that the Maurer--Cartan invariants $\{\phi,\psi,\kappa,\tau,\sigma\}$ form a generating set.
We will show that $\{\kappa, \tau, \sigma\}$ can be written in terms of
$\{\phi,\psi\}$ and their derivatives. We write those as $\phi_{{ij}}$
to mean ${\mathcal D}^i_1{\mathcal D}^j_2( \phi)$
and similarly for $\psi$, $\kappa$, $\tau$, $\sigma$.

The  non-zero syzygies of  Theorem~\ref{mcsyz:th} are:
\begin{alignat*}{3}
& \Delta_{7}:\quad &&
\phi_{{10}}+\psi_{{01}}-2\,\tau+2\,\kappa = 0,&\\
& \Delta_{8}:&&
\sigma_{{01}}-\tau_{{10}}-\tfrac{1}{2}\,\phi -2\,\phi \,\sigma -2\,\psi \,\tau =0,&\\
& \Delta_{9}: && \sigma_{{10}}+\kappa_{{01}}-2\,\phi \,\kappa
    +2\,\psi \,\sigma =0,&\\
& \Delta_{10}:& & \tfrac{1}{2}\,\phi_{{10}}-\tau
   +\psi \,\phi =0.&
\end{alignat*}
The syzygies  $ \Delta_{10}$
and $ { \Delta_{7}}+2\,{\Delta_{10}}$
 allow us to rewrite $\tau$ and $\kappa$ in terms of $\phi, \psi$, namely:
\[
\kappa=-\tfrac{1}{2}\,\psi_{{01}}+\psi\phi,
\qquad
\tau=\tfrac{1}{2}\,\phi_{{10}}+\psi\phi,
\]
while the following combination
\[
2\,{\mathcal D}_2({\Delta_{9}})-2\,{\mathcal D}_1({\Delta_{8}})
+4\,\sigma{\Delta_{7}}-6\,\psi{\Delta_{8}}
-6\,\phi{\Delta_{9}}-2\,{\Delta_{10}}
\]
allows to express $\sigma$ in terms of
$\phi$, $\psi$, $\tau$, $\kappa$ and their derivatives:
\[
 \sigma =
{\frac {\tau_{{20}}+\kappa_{{02}} = 5\,\psi\tau_{{10}}-5\,\phi\kappa_{{01}}
    +2\,\psi_{{10}}\tau-2\,\phi_{{01}}\kappa +  6\,{\phi}^{2}\kappa
     +(6\,{\psi}^{2}+1)\,\tau
       +\frac 12\,\psi\phi}{4(\kappa-\tau)}}.
\]
Observe that this exhibits a singular behavior at \textit{umbilic points} where $\kappa = \tau $.

Finally, since the generating invariants $\{\phi,\psi\}$
are, up to a constant multiple, commutator invariants, we can generate one from the other by the same formulas \eqref{psi:eq},~\eqref{phi:eq}, under the assumption that
\eqref{p1p20:eq} holds.

\section{Projective surfaces}

The inf\/initesimal generators of the projective action of ${\rm PSL}(4)$ on ${\mathbb R}^3$ are
\begin{gather*} \frac{\partial}{\partial x}, \qquad \frac{\partial}{\partial y},\qquad \frac{\partial}{\partial u},
\qquad
 x \frac{\partial}{\partial x}, \qquad y \frac{\partial}{\partial x},\qquad
  u  \frac{\partial}{\partial x}, \\
  x \frac{\partial}{\partial y}, \qquad y \frac{\partial}{\partial y},\qquad u  \frac{\partial}{\partial y}, \qquad
  x \frac{\partial}{\partial u}, \qquad y \frac{\partial}{\partial u},\qquad u  \frac{\partial}{\partial u},\\
x^2 \frac{\partial}{\partial x} + xy \frac{\partial}{\partial y}+xu\frac{\partial}{\partial u}, \qquad
x y \frac{\partial}{\partial x} + y^2  \frac{\partial}{\partial y}+ y u\frac{\partial}{\partial u}, \qquad
x u \frac{\partial}{\partial x} + y u  \frac{\partial}{\partial y}+ u^2 \frac{\partial}{\partial u} .
\end{gather*}
The generic prolonged orbit dimensions are $r_0=3$, $r_1=5$, $r_2=8$, $r_3=12$ and $r_4=15 = \dim {\rm PSL}(4)$, and so the stabilization order is $s=4$.

We adopt the same strategy as in previous section to show that all
the dif\/ferential invariants are generated by a single fourth order
dif\/ferential invariants. The computations and formulae are nonetheless
more challenging.

The section implicitly used in \cite{Tressedi} is:
\begin{gather}
x=y=u=u_x=u_y=u_{xx}=u_{yy}=u_{xxy}=u_{xyy}=u_{xxxy}=u_{xxyy}=u_{xyyy}=0,\nonumber\\
u_{xy} = u_{xxx} = u_{yyy} =1.\label{prcs:eq}
\end{gather}
Thus, there are two basic fourth order dif\/ferential invariants:
\[
I_{40} = \iota(u_{xxxx}),\qquad I_{04} = \iota(u_{yyyy}),
\]
and $6$ of order $5$, given by invariantization of the f\/ifth order jet coordinates.  Theorem~\ref{minorder:th} implies that the invariants $\{I_{40},I_{04},I_{41},I_{32},I_{23},I_{14}\}$ forms a generating set of dif\/ferential invariants.

The Maurer--Cartan matrix \eqref{MCm:eq} is
\begin{gather*}
R =
- \left(\!\begin{array}{ccccccccccccccc}
1 & 0 & 0 & -2\,\psi & 0 & \kappa
& -\tfrac{1}{2} & -\psi & \tau & 0 & 1 & -3\,\psi
& -\tau & \tfrac{1}{4}-\kappa & \tfrac{1}{2}\,\sigma-\tfrac{3}{8}\,\psi
\\
0 & 1 & 0
& \phi & -\tfrac{1}{2} & \sigma
 & 0 & 2\,\phi & \eta
& 1 & 0 & 3\,\phi &
\tfrac{1}{4}-\eta & -\sigma & \tfrac{3}{8}\,\phi+\tfrac{1}{2}\,\tau
\end{array}\!\right)\!,\!
\end{gather*}
where
\begin{gather*}
\phi=-\tfrac{1}{3}\,I_{{04}},\qquad \psi=\tfrac{1}{3}\,I_{{40}},\qquad \eta=-\tfrac{1}{2}\,I_{{14}}-\tfrac{1}{4},
\\
\tau=-\tfrac{1}{2}\,I_{{23}}+\tfrac{1}{4}\,I_{{04}},\qquad
\sigma=-\tfrac{1}{2}\,I_{{32}}+\tfrac{1}{4}\,I_{{40}},\qquad
\kappa=-\tfrac{1}{2}\,I_{{41}}-\tfrac{1}{4} .
\end{gather*}
By Theorem~\ref{mcgen:th} $\{\phi,\psi,\tau,\sigma, \kappa \}$
form a generating set of dif\/ferential invariants.
The invariant derivations satisfy the commutation rule;
\[
   [{\mathcal D}_2,{\mathcal D}_1] = \phi\, {\mathcal D}_1 + \psi\,{\mathcal D}_2
\]
and so $\phi$, $\psi $ are the commutator invariants.

The nonzero syzygies of Theorem~\ref{mcsyz:th} of the generating set
$\{\phi,\psi, \eta,\sigma, \tau, \kappa\}$ are given by:
\begin{alignat*}{3}
& \Delta_{4}: \quad &&
\phi_{{10}}+2\,\psi_{{01}}+2\,\eta-\phi\,\psi-\tfrac{1}{2} =0,&
\\
& \Delta_6: &&
\sigma_{{10}}-\kappa_{{01}}-\tfrac{3}{8}\,\phi+3\,\phi\,\kappa+2\,\psi\,\sigma=0,&
\\
& \Delta_{8}: & &
2\,\phi_{{10}}+\psi_{{01}}-2\,\kappa+\phi\,\psi+ \tfrac{1}{2}=0,&
\\
& \Delta_{9}: &&
\eta_{{10}}-\tau_{{01}}-\tfrac{3}{8}\,\psi+2\,\phi\,\tau+3\,\psi\,\eta=0,&
\\
& \Delta_{12}: & &
{\Delta_{4}}+{ \Delta_{8}},
\qquad
\Delta_{13}: \ \
- \Delta_{9},
\qquad
\Delta_{14}: \ \
 - \Delta_{6} ,&
\\
& \Delta_{15}: & &
\tfrac{1}{2}\,\tau_{{10}}-\tfrac{1}{2}\,\sigma_{{01}}
+\tfrac{3}{8}\,\phi_{{10}}+\tfrac{3}{8}\,\psi_{{01}}-\tfrac{1}{4}\,\kappa+\tfrac{1}{4}
\,\eta+2\,\phi\,\sigma+2\,\psi\,\tau=0.&
\end{alignat*}
From ${\Delta_{4}}$ and ${ \Delta_{8}}$ we immediately obtain:
\[
\eta=
\tfrac{1}{4}-\tfrac{1}{2}\,\phi_{{10}}
-\psi_{{01}}+\tfrac{1}{2}\,\phi\psi,
\qquad
\kappa=
\tfrac{1}{4}+\phi_{{10}}+\tfrac{1}{2}\,\psi_{{01}}+\tfrac{1}{2}\,\phi\psi.
\]
Let $P_1$, $P_2$, $P_3$ be the dif\/ferential polynomials obtained from
 $\Delta_6$, $\Delta_9$, $\Delta_{15}$ after substitution of~$\kappa$ and $\tau$:
\begin{gather*}
P_1=-\tfrac{1}{2}\,\tau_{{10}}+\tfrac{1}{2}\,\sigma_{{01}}-2\,\phi\,\sigma-2\,\tau\,\psi,
\\
P_2=\tfrac{1}{2}\,\phi_{{20}}+\psi_{{11}}-\tfrac{1}{2}\,\phi_{{10}}
\psi-\tfrac{1}{2}\,\phi\psi_{{10}}+\tau_{{01}}-\tfrac{3}{8}\,\psi
-2\,\phi\,\tau+\tfrac{3}{2}\,\psi\,\phi_{{10}}+3\,\psi\,\psi_{{01}}
-\tfrac{3}{2}\,\phi\,{\psi}^{2},
\\
P_3=-\sigma_{{10}}+\phi_{{11}}+\phi\phi_{{10}}
+\tfrac{3}{2}\,\psi\phi_{{01}}+\tfrac{1}{2}\,\psi_{{02}}+\tfrac{1}{2}\,\phi
\psi_{{01}}-\tfrac{3}{8}\,\phi-3\,\phi\,\phi_{{10}}
-\tfrac{3}{2}\,\phi\,\psi_{{01}}\\
\phantom{P_3=}{} -\tfrac{3}{2}\,{\phi}^{2}\psi-2\,\psi\,\sigma.
\end{gather*}

To obtain $\tau$ and $\sigma$ we proceed with a \emph{differential elimination}
 \cite{diffalg,hubert03d,ncdiffalg,hubert05}  on $\{P_1,P_2,P_3\}$. We use a ranking where
\begin{gather*}
\psi<\phi<\psi_{{01}}<\phi_{{01}}<\psi_{{10}}<\phi_{{10}} <
\psi_{{02}}<\psi_{{11}}<\phi_{{11}}<\phi_{{20}}<\>\cdots\\ \qquad
\cdots\><\tau<\sigma <
\tau_{{01}}<\sigma_{{01}}<\tau_{{10}}<\sigma_{{10}}<\tau_{{02}} <
\sigma_{{02}}<\tau_{{11}}<\sigma_{{11}}<\tau_{{20}}<\sigma_{{20}}<\>\cdots\>.
\end{gather*}
For this ranking, the leaders of $P_1$, $P_2$, $P_3$ are, respectively,
 $\tau_{{10}}$, $\tau_{{01}}$, $\sigma_{{10}}$.

We f\/irst form the \emph{$\Delta$-polynomial} (cross-derivative) of
$P_1$ and $P_2$ and \emph{reduce} it with respect to
$\{P_1,P_2,P_3\}$. We obtain a polynomial $P_4$ with leader
$\sigma_{{02}}$. We then take the $\Delta$-polynomial of~$P_3$ and
$P_4$ and reduce it with respect to $\{P_1,P_2,P_3,P_4\}$ to obtain a
dif\/ferential polynomial $P_5$ with leader $\sigma_{{01}}$.
On one hand, if we reduce now $P_4$ by $\{P_1,P_2,P_3,P_5\}$ we obtain
a dif\/ferential polynomial~$P$
 with leader $\sigma$.
On the other hand, if we form the $\Delta$-polynomial of~$P_3$ and~$P_5$, reduce it by $\{P_1,P_2,P_3,P_5\}$ we obtain a dif\/ferential
polynomial $Q$ with leader $\sigma$. The polynomial~$P$ and~$Q$ are linear in $\sigma$ and $\tau$ so that we can
solve for those two invariants in terms of $\phi$, $\psi$ and their
derivatives.  The explicit formulas are rather long (available from the authors on request), but not particularly enlightening.  We conclude that the commutator invariants~$\phi$,~$\psi$ form a~generating set.  Finally, we can use  either \eqref{psi:eq} or \eqref{phi:eq}, to generate one commutator invariant from the other, and thereby establish Theorem~\ref{projective:th}.

\subsection*{Acknowledgements}

This research was initiated during the f\/irst author's visit to the Institute for Mathematics and its Applications (I.M.A.) at the University of Minnesota during 2007--2008 with additional support from the Fulbright visiting scholar program.
The second author is supported in part by NSF Grant DMS 05--05293.

\pdfbookmark[1]{References}{ref}

\LastPageEnding

\end{document}